\documentclass[11pt]{amsart}
\usepackage{amssymb,amsmath,amsthm,verbatim,enumitem}
\usepackage{mathtools}
\usepackage{fullpage}
\usepackage{xcolor}
\usepackage{url}
\usepackage{array} 
\usepackage{mathrsfs}
\usepackage[all]{xy}
  \SelectTips{cm}{10}
  \everyxy={<2.5em,0em>:}
\usepackage{tikz}
\usepackage{picinpar} % for pictures in paragraphs

\usepackage[new]{old-arrows}
\usepackage{fancyhdr}
\usepackage{ifpdf}
  \ifpdf
    \usepackage{hyperref} % messes up dvi output for some reason
  \else
    
    \newcommand{\href}[2]{#2}
  \fi

\theoremstyle{plain}
  \newtheorem{lemma}[equation]{Lemma}
  \newtheorem{proposition}[equation]{Proposition}
  \newtheorem{theorem}[equation]{Theorem}
  \newtheorem{corollary}[equation]{Corollary} 
   
    \newtheorem{question}[equation]{Question}
    \newtheorem{conjecture}[equation]{Conjecture}

\theoremstyle{definition}
  \newtheorem{definition}[equation]{Definition}

\theoremstyle{remark}

\renewcommand{\thesection}{\arabic{section}}
\renewcommand{\theequation}{\thesection.\arabic{equation}}

 \DeclareFontFamily{U}{manual}{}
 \DeclareFontShape{U}{manual}{m}{n}{ <->  manfnt }{}
 \newcommand{\manfntsymbol}[1]{%
    {\fontencoding{U}\fontfamily{manual}\selectfont\symbol{#1}}}

\makeatletter
   \@addtoreset{section}{part}
   \@addtoreset{equation}{section}
   \@addtoreset{footnote}{section}

    {\hspace*{\fill}$\lrcorner$\endgraf\endgroup\end{trivlist}}
 \newenvironment{example}[1][]{
   \refstepcounter{equation}
   \begin{proof}[Example~\theequation%
   \@ifnotempty{#1}{ (#1)}.]
   }
  {\end{proof}}
\makeatother

  \DeclareFontFamily{OT1}{pzc}{}
  \DeclareFontShape{OT1}{pzc}{m}{it}{<-> s * [1.100] pzcmi7t}{}
  \DeclareMathAlphabet{\mathpzc}{OT1}{pzc}{m}{it}

\newif\ifhascomments \hascommentstrue
\ifhascomments
  \newcommand{\dan}[1]{{\color{red}[[\ensuremath{\bigstar\bigstar\bigstar} #1]]}}
  \newcommand{\matt}[1]{{\color{red}[[\ensuremath{\spadesuit\spadesuit\spadesuit} #1]]}}
  
\else
  \newcommand{\dan}[1]{}
  \newcommand{\matt}[1]{}
\fi

\newcommand{\<}{\langle}
\renewcommand{\>}{\rangle} % old \> is a spacing command
\renewcommand{\AA}{\mathbb{A}}

\DeclareMathOperator{\Aut}{\ensuremath{\mathcal{A}\kern-.125em\mathpzc{ut}}}

\newcommand{\CC}{\mathbb C}

\DeclareMathOperator{\Endo}{\ensuremath{\mathcal{E}\kern-.125em\mathpzc{nd}}}

\DeclareMathOperator{\Hom}{\ensuremath{\mathcal{H}\kern-.125em\mathpzc{om}}}

\DeclareMathOperator{\lra}{\longrightarrow}

\DeclareMathOperator{\ord}{ord}

\newcommand{\PP}{\mathbb{P}}

\newcommand{\QQ}{\mathbb Q}
\newcommand{\RR}{\mathbb R}

\renewcommand{\setminus}{\smallsetminus}

\DeclareMathOperator{\Span}{Span}

\DeclareMathOperator{\stab}{Stab}

\newcommand{\T}{\mathcal T}

\newcommand{\ZZ}{\mathbb{Z}}

 \def\ari[#1]{\ar@{^(->}[#1]}
 \def\are[#1]{\ar[#1]^{\txt{\'et}}}
 \def\areh[#1]{\ar[#1]|{\txt{$H$-eq}}^{\txt{\'et}}}
 \def\ars[#1]{\ar@{->>}[#1]}
 \newcommand{\dplus}{\ar@{}[d]|{\mbox{$\oplus$}}}
 \newcommand{\dtimes}{\ar@{}[d]|{\mbox{$\times$}}}

\newcolumntype{L}{>{$}l<{$}}

\title{What fraction of an $S_n$-orbit can lie on a hyperplane?}

\author{Jiahui Huang}
\address{University of Waterloo \\
Department of Pure Mathematics \\
Waterloo, Ontario \\
Canada  N2L 3G1}
\email{j346huang@edu.uwaterloo.ca}

\author{David McKinnon}
\address{University of Waterloo \\
Department of Pure Mathematics \\
Waterloo, Ontario \\
Canada  N2L 3G1}
\email{dmckinnon@uwaterloo.ca}

\author{Matthew Satriano}
\address{University of Waterloo \\
Department of Pure Mathematics \\
Waterloo, Ontario \\
Canada  N2L 3G1}
\email{msatrian@uwaterloo.ca}

\thanks{The last two authors were partially supported by Discovery Grants from the Natural Sciences and Engineering Research Council.}

\begin{document}

\begin{abstract}
Consider the $S_n$-action on $\RR^n$ given by permuting coordinates. This paper addresses the following problem: compute $\max_{v,H} |H\cap S_nv|$ as $H\subset\RR^n$ ranges over all hyperplanes through the origin and $v\in\RR^n$ ranges over all vectors with distinct coordinates that are not contained in the hyperplane $\sum x_i=0$. We conjecture that for $n\geq3$, the answer is $(n-1)!$ for odd $n$, and $n(n-2)!$ for even $n$. We prove that if $p$ is the largest prime with $p\leq n$, then $\max_{v,H} |H\cap S_nv|\leq \frac{n!}{p}$. In particular, this proves the conjecture when $n$ or $n-1$ is prime.
\end{abstract}

\maketitle
\tableofcontents

\section{Introduction}
Given a linear action of a Lie group $\mathcal{G}$ on a finite-dimensional vector space $W$, a question of central importance is to determine when the quotient $W/\mathcal{G}$ is smooth. This problem and variants of it have a long history in invariant theory with fundamental classification results having been obtained in \cite{ShTo:54,Ch:55,KPV:76,Sch:78b,Sch:78a,AdGo:79,Lit:89}. In a recent preprint \cite{EdSa:19}, \emph{cf.}~\cite{Sch:94}, Edidin and the third author considered the problem of giving an effective group theoretic characterization for when $W/\mathcal{G}$ is smooth, and related it to a variant of 
%this to a variant\footnote{\matt{they were concerned with $G=S_n$, $V$ is the (irreducible) standard representation, and they had a \emph{fixed} finite union $Gv_1\cup\dots\cup Gv_r$ where $|Gv_i|$ is not necessarily equal to $|G|$.}
the following concrete question.

\begin{question}
\label{q:hyperplane}
Let $G$ be a finite group and $V$ a finite-dimensional $G$-representation over a field $k$. Let $V=\bigoplus_i V_i$ be the decomposition into irreducible representations. What is
\[
\max_{v,H} |H\cap Gv|
\]
as $H\subset V$ ranges over all hyperplanes through the origin, and $v\in V\setminus\bigcup_i V_i$ ranges over all vectors whose orbit satisfies $|Gv|=|G|$?
\end{question}

In \cite{EdSa:19}, the authors were primarily concerned with the case where $G=S_n$ and $k=\RR$, and obtained bounds sufficient for their purposes, but the question of a general bound remained. We make the following conjecture:

\begin{conjecture}
\label{conj:max-bound}
Let $n\geq3$. As $v$ ranges over all vectors in $\RR^n$ with distinct coordinates not in the hyperplane $\sum_i x_i=0$, and as $H\subset\RR^n$ ranges over all hyperplanes through the origin, we have
\[
\max_{v,H} |H\cap S_nv|=
\begin{cases}
(n-1)!,& n \mathrm{\ is\ odd}\\
n(n-2)!,& n \mathrm{\ is\ even}
\end{cases}
\]
%where $S_nv$ denotes the orbit of $v$ under the permutation action of $S_n$ on $\RR^n$.
\end{conjecture}

Let us motivate how these specific bounds arise.

\begin{example}
\label{ex:bound-is-attained-general}
Given any $v=(c_1,\dots,c_n)$ with distinct coordinates, consider the hyperplane $H$ whose normal vector is $(c_n,c_n,\dots,c_n,-\sum_{i=1}^{n-1}c_i)$. Then $H$ contains $(c_{\sigma(1)},\dots,c_{\sigma(n-1)},c_n)$ for all $\sigma\in S_{n-1}$, so $|H\cap S_nv|\geq(n-1)!$.
\end{example}

\begin{example}
\label{ex:bound-is-attained-even}
Let $n\geq3$. Consider the vector $v=(1,2,\dots,n)$ and the hyperplane $H$ with normal vector $(-\sum_{i=2}^{n-1}i,-\sum_{i=2}^{n-1}i,1+n,\dots,1+n)$. 
%\[H=(-\sum_{i=2}^{n-1}i,-\sum_{i=2}^{n-1}i,1+n,\dots,1+n)^\perp.\]
Then $H$ contains every element of $S_nv$ whose first two coordinates sum to $1+n$. When $n$ is odd, there are $(n-1)!$ such elements of $S_nv$. When $n$ is even, there are $n(n-2)!$ such elements.
% $\max_H |H\cap S_nv|\geq n(n-2)!$.
\end{example}

By the above two examples, the bounds in Conjecture~\ref{conj:max-bound} are the smallest possible. The main result of this paper is:

\begin{theorem}
\label{thm:conj-for-primes}
Let $n\geq 2$ and let $p$ be the largest prime with $p\leq n$.  Then 
\[
\max_{v,H} |H\cap S_nv|\leq \frac{n!}{p}.
\]
In particular, if $n=p$ or $n=p+1$, then Conjecture \ref{conj:max-bound} is true.
\end{theorem}

The proof of Theorem~\ref{thm:conj-for-primes} involves tools from algebraic geometry, representation theory, combinatorics, and graph theory. The proof proceeds in several steps. 
%First, by Proposition \ref{thm:conj-for-primes}, it is enough to handle the case where $n=p$ is prime. 
Using techniques from algebraic geometry, we reduce the problem to one concerning intersections of hyperplanes with a specific curve $C$. We then divide the proof into two cases depending on whether or not the irreducible components $C_i$ of $C$ have dihedral stabilizers. We handle the non-dihedral case using techniques from combinatorics and representation theory. The dihedral case is the most involved. We construct a graph whose vertices are the irreducible components $C_i\subset H$. Assuming the existence of a hyperplane $H$ that violates Theorem~\ref{thm:conj-for-primes}, we show the existence of a vertex $C_0$ whose neighbors have large degree relative to $C_0$. A careful analysis of the second order neighborhood of $C_0$ yields a contradiction.

Additionally, we prove the following two results. The first shows that the conjecture holds for generic $v$ and the second gives an inductive statement, showing that the case of even $n$ follows from that of odd $n$.

\begin{proposition}
\label{prop:conj-for-generic-v}
Let $n\geq2$. There is a nonempty Zariski open subset $U$ of $\RR^n$ such that for any $v\in U\subset\RR^n$, we have $\max_H |H\cap S_nv|=(n-1)!$ as $H$ ranges over all hyperplanes in $\RR^n$. 
\end{proposition}

\begin{proposition}
\label{prop:conj-odd-->even}
Let $k\leq n$ be positive integers.  If $\max_H |H\cap S_nv|\leq n!/k$, then for all $m\geq n$, we have $\max_H |H\cap S_mv|\leq m!/k$.

In particular, if Conjecture \ref{conj:max-bound} holds for an odd number $n$, then it is also holds for $n+1$.
\end{proposition}

\section*{Acknowledgments}
It is our pleasure to thank Jason Bell, Ilya Bogdanov, Dan Edidin, Matt Kennedy, Heydar Radjavi, and Jerry Wang for helpful conversations. This paper is the outcome of an NSERC-USRA project; we thank NSERC for their support through the USRA program.

\section{Proof of Propositions \ref{prop:conj-for-generic-v} and \ref{prop:conj-odd-->even}}
\label{sec:props-from-intro}

We begin this section by analyzing the behaviour of $\max_H |H\cap S_nv|$ where $v$ is a generic vector:

\begin{proof}[{Proof of Proposition \ref{prop:conj-for-generic-v}}]
By Example \ref{ex:bound-is-attained-general}, we know that for every $v$ with distinct coordinates, there exists a hyperplane $H$ with $|H\cap S_nv|\geq(n-1)!$. So, it remains to show that for generic $v$ we have $|H\cap S_nv|\leq(n-1)!$ for every hyperplane $H$. Let $v=(x_1, \dots, x_n)$ where $x_1,\dots x_n$ are indeterminates, and let $\Omega$ be the set of all subsets of $S_nv$ consisting of $(n-1)! + 1$ elements. For every $\omega\in\Omega$, let $M_\omega$ be the matrix whose columns are the vectors in the set $\omega$ (with some ordering of the set $\omega$ whose choice will not affect the proof). Then we must show that for every $\omega\in\Omega$, the $n \times n$ minors of $M_\omega$ do not simultaneously vanish. Let $V_\omega\subseteq \AA^n$ be the variety defined by the simultaneous vanishing of the $n \times n$ minors of $M_\omega$. We need to show $\bigcup_{\omega\in\Omega} V_\omega\neq\AA^n$, so it is enough to show that for each $\omega\in\Omega$, there exists some $v\in\AA^n$ with $v\notin V_\omega$. 

We prove this by induction. When $n=2$, we must have $\omega=\{(x_1,x_2),(x_2,x_1)\}$, so $v=(1,0)$ will suffice.

Now suppose $n>2$. Consider the appearance of $x_n$ in the rows of $M_\omega$. If $x_n$ shows up at least once in each row, let $v=(0,\dots,0,1)$; then the column vectors in $M_\omega$ will contain the standard basis vectors, so the $n\times n$ minors will not vanish. If $x_n$ does not appear in some row, then it only occurs in at most $n-1$ of the rows; hence, some row contains at least $\frac{(n-1)!+1}{n-1}>(n-2)!$ copies of $x_n$. By permuting rows of $M_\omega$, we may assume there is a subset of $\omega'\subset\omega$ such that $|\omega'|=(n-2)!+1$ and every vector in $\omega'$ has $x_n$ as its last entry.

By induction, we can specialize the variables $x_1,\dots,x_{n-1}$ to be distinct real numbers in such a way that the column vectors in $\omega'$ span a space of dimension at least $n-1$. Choose $x_n$ so that $\sum_i x_i\neq 0$ and $x_n\neq x_i$ for $i=1,\dots,n-1$. Since the column vectors of $M_{\omega'}$ have the same last coordinate, they are all contained in the hyperplane $H$ constructed in Example \ref{ex:bound-is-attained-general}. We have therefore shown that if the $n\times n$ minors of $M_\omega$ vanish, then the span of the column vectors of $M_\omega$ is $H$. However, since $|\omega|>(n-1)!$ and the $x_i$ are distinct real numbers, some column vector of $M_\omega$ must have last coordinate not equal to $x_n$; this vector is not in $H$ and therefore the $n\times n$ minors of $M_\omega$ do not simultaneously vanish.
\end{proof}

We turn next to Proposition \ref{prop:conj-odd-->even}.

\begin{proof}[{Proof of Proposition \ref{prop:conj-odd-->even}}]
We prove the result by induction on $n$. We assume there exists $k\leq n-1$ such that for all $w\in\RR^{n-1}$ with distinct coordinates not summing to $0$, and all hyperplanes $H'\subset\RR^{n-1}$, we have $|S_{n-1}w\cap H'|\leq (n-1)!/k$. Now, let $v=(v_1,\dots,v_n)\in\RR^n$ with distinct coordinates not summing to $0$. Suppose there exists $T\subseteq S_nv$ and a hyperplane $H\subset\RR^n$ such that
\[
|T|=|T\cap H|>\frac{n!}{k}.
\]
Since $S_n$ is the disjoint union of the cosets $(in)S_{n-1}$ for $1\leq i\leq n$, there exists $i$ with $|T\cap (in)S_{n-1}v\cap H|>(n-1)!/k$. Relabeling the coordinates of $\RR^n$ if necessary, we can assume $i=n$. Let
\[
U=T\cap S_{n-1}v
\]
and $\pi\colon\RR^n\to\RR^{n-1}$ be the projection map $\pi(x_1,\dots,x_n)=(x_1,\dots,x_{n-1})$. Note that $\pi(v)$ has distinct coordinates and that $\pi(U)\subseteq S_{n-1}\pi(v)$. Moreover, $\pi|_U\colon U\to\pi(U)$ is a bijection, so $|\pi(U)|>(n-1)!/k$.

If $v_1+\ldots+v_{n-1}\neq 0$, then by induction, $\pi(U)$ is not contained in a hyperplane, and must therefore span $\RR^{n-1}$. As a result, $\Span(U)$ is either a hyperplane or $\RR^n$. Notice that $U$ is contained in the hyperplane $H'$ given by $cx_n=v_n(x_1+\ldots+x_{n-1})$ with $c=v_1+\ldots +v_{n-1}$, i.e.~the hyperplane constructed in Example \ref{ex:bound-is-attained-general}. Thus, $\Span(U)=H'$ and hence $H'=H\supset T$. However, $|S_nv\cap H'|=(n-1)!$ which implies $(n-1)!\geq |T|>n!/k$, a contradiction.

If $v_1+\ldots+v_{n-1}=0$, then $\pi(U)$ is contained in the hyperplane $x_1+\ldots+x_{n-1}=0$. Let $w=\pi(v)+(1,\dots,1)\in\RR^{n-1}$. Then the coordinates of $w$ are distinct and do not sum to $0$, so by induction, $\pi(U+(1,\dots,1))$ spans $\RR^{n-1}$. In particular, $\pi(U)$ spans the hyperplane $x_1+\ldots+x_{n-1}=0$. This implies that $U$ is not contained in any affine space of dimension less than $n-2$. Notice that $U$ is contained in the affine space $A$ given by $x_1+\ldots+x_{n-1}=x_n-v_n=0$, and that $A$ has dimension exactly $n-2$. Note further that $A$ is not a linear space since $v_n\neq0$, and so $U$ is not contained in any linear space of dimension $n-2$. Thus, $U$ must span an $(n-1)$-dimensional space and since $U$ is contained in the hyperplane $H'$ given by $x_1+\ldots+x_{n-1}=0$, we must have $\Span(U)=H'$, and so $H=H'$. However, we see $(in)S_{n-1}v\cap H'=\varnothing$ for all $i\neq n$. Thus, we again find $(n-1)!\geq |H'\cap T|=|T|>n!/k$, a contradiction.
\end{proof}

As a result of Proposition \ref{prop:conj-odd-->even}, we only need to consider the case when $n=p$ for the proof of Theorem \ref{thm:conj-for-primes}.

\section{An analysis via algebraic geometry}
\label{sec:AG-input}

Let $v=(c_1,\dots,c_n)\in\RR^n$ with distinct coordinates. Consider the elementary symmetric functions $e_k(x_1,\ldots,x_n)$ for $1\leq k\leq n$, and the following system of equations:
\begin{align*}
e_1(x_1,\ldots,x_n) &= e_1(c_1,\ldots,c_n) \\
& \vdotswithin{=} \\
% &\vdots  \\
%\vdots\hspace*{.5in} &\vdots \hspace*{.5in}\vdots \\
e_t(x_1,\ldots,x_n) &= e_t(c_1,\ldots,c_n) 
\end{align*}
The set $S_nv$ is precisely the solution set of this system when $t=n$.  

\begin{definition}
\label{def:elem-sym-curve}
Let $C\subset\PP^n_\CC$ be the algebraic variety cut out by the above system of equations where we take $t=n-1$. We refer to $C$ as the \emph{elementary symmetric curve associated to} $v$.
\end{definition}

%If we set $t=n-1$, then the solution set of the system is a curve $C$ (possibly reducible).  By Bezout's Theorem, since the polynomial $e_k$ has degree $k$, the curve $C$ has degree $(n-1)!$.

The elementary symmetric curve plays a fundamental role in this paper. Throughout the rest of this section, we let $C$ be the elementary symmetric curve associated to $v$ and let
\[
C=C_1\cup\dots\cup C_r
\]
be the decomposition of $C$ into its irreducible components.

%\begin{lemma}\label{l:bezout-and-its-a-curve}Each irreducible component $C_i$ has dimension 1. Furthermore, if \end{lemma}

If there is a hyperplane $H$ which contains many conjugates of $v$, then it will have a large intersection with $C$.  If $H$ intersects $C$ properly (that is, in a finite set of points), then $H$ cannot intersect $C$ in more than $(n-1)!$ points, and therefore $H$ cannot contain more than $(n-1)!$ conjugates of $v$.  

Write $C=C_1\cup\ldots\cup C_r$ as a union of irreducible curves.  (Note that $C$ cannot have any components of dimension greater than one, because its intersection with the hypersurface $e_n(x_1,\ldots,x_n)=e_n(c_1,\ldots,c_n)$ is a finite set of points, namely $S_nv$.)  If $H$ contains more than $(n-1)!$ conjugates of $v$, then it must contain some irreducible component $C_i$ of $C$.

\begin{lemma}
\label{lem:transitive-action-components}
We have the following properties:
\begin{enumerate}
\item\label{lem:transitive-action-components::dim1} Each $C_i$ has dimension 1.
\item\label{lem:transitive-action-components::hyp-n-1} If $H\subset\RR^n$ is a hyperplane that intersects $C$ properly, i.e.~in a finite set of points, then $|H\cap C|\leq (n-1)!$.
\item\label{lem:transitive-action-components::hyp-contain-Ci} If $H\subset\RR^n$ is a hyperplane satisfying $|H\cap C|>(n-1)!$, then $H$ contains some $C_i$.
\end{enumerate}
\end{lemma}
\begin{proof}
Notice that the intersection of $C$ with the hypersurface $e_n(x_1,\ldots,x_n)=e_n(c_1,\ldots,c_n)$ is a finite set of points, namely $S_nv$. Since intersecting with a hypersurface decreases dimension by at most 1, we see each $C_i$ has dimension at most 1. On the other hand, $C$ is defined as the intersection of $n-1$ hypersurfaces, so each $C_i$ has dimension at least 1. This proves (\ref{lem:transitive-action-components::dim1}).

Statements (\ref{lem:transitive-action-components::hyp-n-1}) and (\ref{lem:transitive-action-components::hyp-contain-Ci}) follows immediately from B\'ezout's Theorem since $C$ is intersection of hypersurfaces of degrees $1,2,\dots,n-1$, and hence has degree $(n-1)!$.
\end{proof}

\begin{lemma}
\label{lem:irred-comp-same-degree}
%\item\label{lem:transitive-action-components::transitive} 
The $S_n$-action on $\RR^n$ induces a transitive action on the set of irreducible components $\{C_1,\dots,C_r\}$. Moreover, for each $i$, we have $\deg(C_i)=\frac{(n-1)!}{r}$ and $|C_i\cap S_nv|=\frac{n!}{r}$.
%Components of $C$ have degree $(n-1)!/r$. Each $C_i$ contains exactly $n!/r$ elements of $S_nv$.
\end{lemma}
\begin{proof}
Since $Y\cap C=S_nv$ consists of $n!=\deg(Y)\deg(C)$ distinct points, $Y$ intersects $C$ properly and transversely. In particular, $Y$ cannot intersect $C$ at any point of intersection of two irreducible components of $C$. So, we find
\[
\deg(Y)\deg(C)=|Y\cap C|=\sum_i |Y\cap C_i|\leq\sum_i\deg(Y)\deg(C_i)=\deg(Y)\deg(C)
\]
from which we see
\[
|C_i\cap S_nv|=|Y\cap C_i|=\deg(Y)\deg(C_i)
\]
for every $i$. It follows that $C_i\cap S_nv\neq\varnothing$ and so the $S_n$-action on $\{C_1,\dots,C_r\}$ is transitive. As a result, each $C_i$ has the same degree and contains the same number of elements of $S_nv$, so we must have $\deg(C_i)=\frac{(n-1)!}{r}$ and $|C_i\cap S_nv|=\frac{n!}{r}$.
%Let $O$ be the orbit of the $S_n$ action on the $C_i$ that contains $v$.  Then every conjugate of $v$ lies on some $C_i$ in $O$.  Let $C'=\bigcup_{C_i\in O}C_i$, and let $Y$ be the hypersurface given by $e_n(x_1,\ldots,x_n)=s_n$.  Then $C'\cap Y=C\cap Y=S_nv$.  Since $Y$ intersects $C$ properly and transversely in $n!$ points, $Y$ cannot intersect $C$ at a point of intersection of two irreducible components, and so any $C_i$ not in $O$ cannot contain any conjugate of $v$, and therefore must be disjoint from $Y$.  This is impossible by Bezout's Theorem, and so there are no $C_i$ not contained in $O$, and the action of $S_n$ on the $C_i$ is transverse, as desired.
%
%Since the degree of $C$ is the sum of the degrees of the irreducible components and $S_n$ acts on them transitively, each of the irreducible component has the same degree, $(n-1)!/r$.  The hypersurface $Y$ intersects each $C_i$ properly and transversely.  Since $C_i$ has degree $(n-1)!/r$, and since $Y$ has degree $n$, we conclude that $C_i$ has exactly $n!/r$ conjugates of $v$.   
\end{proof}

\begin{lemma}
\label{lem:stab-contain-p-cycle}
Let $n=p$ be prime and $\stab(C_i)$ be the stabilizer of $C_i$ under the $S_p$-action on the set of irreducible components of $C$. Then $\stab(C_i)$ contains a $p$-cycle.
\end{lemma}
\begin{proof}
By Lemma \ref{lem:irred-comp-same-degree}, %Lemma \ref{lem:transitive-action-components} (\ref{lem:transitive-action-components::transitive}), 
$S_p$ acts transitively on the set of irreducible components of $C$. By Lemma \ref{lem:irred-comp-same-degree}, $|\stab(C_i)|=\frac{p!}{r}=p\deg(C_i)$, so $p$ divides $|\stab(C_i)|$. It follows from Cauchy's Theorem that $\stab(C_i)$ contains an element $\pi$ whose order is $p$; since $\pi\in S_p$, it is necessarily a $p$-cycle.
\end{proof}

\begin{corollary}
\label{cor:inv-subspace}
Let $n=p$ be prime and $w=(\zeta,\dots,\zeta^{p})$ where $\zeta=e^{2\pi i/p}$. Then the complex linear span of any irreducible component $C_0$ of $C$ contains $\sigma w$ for some $\sigma\in S_p$.
\end{corollary}
\begin{proof}
Since $\stab(C_0)$ contains a $p$-cycle $\pi$, the complex linear span of $C_0$ contains a subrepresentation of the permutation representation of $\<\pi\>$. This subrepresentation is non-trivial since $v$ has distinct coordinates. Thus, it contains a non-trivial complex irreducible $\<\pi\>$-representation, which is necessarily spanned by $\sigma w$ for some $\sigma\in S_p$.
\end{proof}

We conclude this section with a key lemma used in the proof of Theorem \ref{thm:conj-for-primes}. We know from Lemma \ref{lem:transitive-action-components} %(\ref{lem:transitive-action-components::hyp-n-1}) and 
(\ref{lem:transitive-action-components::hyp-contain-Ci}) that if $H\subset\RR^n$ is a hyperplane with $|H\cap S_nv|>(n-1)!$, then $H$ must contain an irreducible component $C_i$. In the proof of Theorem \ref{thm:conj-for-primes}, we show that for each $C_i\subset H$, there are $n-1$ other irreducible components of $C$ that are not contained in $H$. We then apply the following:
% is to show that if $H$ contains some irreducible component $C_i$ of $C$, then we can find $n-1$ other irreducible components of $C$ that are not contained in $H$.  Once we accomplish this, the following lemma will finish the job.
%We generalize the idea that we used in the last section: for each curve $H$ contains, find $n-1$ other curves that $H$ cannot intersect along the orbit of $v$.

\begin{lemma}
\label{lem:n-1-not-contained}
Let $H\subset\RR^n$ be a hyperplane. Suppose that for each irreducible component $C_i$ of $C$ satisfying $C_i\subset H$, there are irreducible components $C_{i1},\dots,C_{i,{n-1}}$ with the following properties:
\begin{enumerate}
\item[(i)] $H\cap C_{ij}\cap S_nv=\varnothing$ and
\item[(ii)] $C_{ik}=C_{j\ell}$ if and only if $i=j$ and $k=\ell$.
\end{enumerate}
%(i) $H\cap C_{ij}\cap S_nv=\varnothing$ and (ii) $C_{ik}=C_{j\ell}$ if and only if $i=j$ and $k=\ell$.
Then $|H\cap S_nv|\leq (n-1)!$.
\end{lemma}
\begin{proof}
Say $H$ contains exactly $m$ of the irreducible components of $C$.  By Lemma \ref{lem:irred-comp-same-degree} and B\'ezout's Theorem, we then have:
\begin{equation}
\begin{split}
 |H\cap S_nv| &\leq \sum_{C_i\subset H} \frac{n!}{r}\ +\!\!\sum_{\substack{C_k\not\subset H\\ C_k\cap H\cap S_nv\neq \varnothing}} \!\!\!\!\frac{(n-1)!}{r} \\
&\leq m\frac{n!}{r}+(r-mn)\frac{(n-1)!}r\\
&= (n-1)!\qedhere
\end{split}
\end{equation}
\end{proof}

\section{Lemmas concerning 2-cycles, 3-cycles, and 2-2-cycles}
\label{sec:comb-lemmas}

In this section, we collect several results concerning the structure of hyperplanes that simultaneously contain $v$ and $\sigma v$, where $\sigma$ is a 2-cycle, a 3-cycle, or a 2-2-cycle. We also prove Theorem \ref{thm:conj-for-primes} for $p=3,5$.

\begin{lemma}
\label{lem:transpositions-3cycles}
Let $n\geq 3$ and $v=(c_1,\dots,c_n)\in\RR^n$ with distinct coordinates. Let $H=(a_1,\dots,a_n)^\perp$ be a hyperplane containing $v$. If $\tau=(ij)$ is a transposition and $\tau v\in H$, then $a_i=a_j$.

Let $\sigma=(ijk)$ be a 3-cycle and $\pi$ an $n$-cycle. If $H$ contains $\pi^mv$ and $\sigma\pi^mv$ for all $m$, then $a_i=a_j=a_k$.
\end{lemma}
\begin{proof}
By permuting coordinates, we can assume $\tau=(12)$. Then $\tau v-v=(c_2-c_1,c_1-c_2,0,\dots,0)$ is contained in $H$. Since $c_2\neq c_1$, we have $a_1=a_2$.

For the second claim of the lemma, we first permute coordinates to assume $\sigma=(123)^{-1}=(132)$. Then for all $i$, we have\footnote{Recall that if $\epsilon\in S_n$, then the $j$-th coordinate of $\epsilon(v)$ is $c_{\epsilon^{-1}(j)}$.}
\[
d_i:=\sigma\pi^{-i}v-\pi^{-i}v=(c_{\pi^i(2)}-c_{\pi^i(1)},c_{\pi^i(3)}-c_{\pi^i(2)},c_{\pi^i(1)}-c_{\pi^i(3)},0,\dots,0)\in H.
\]
Let $m$ be such that $c_m=\min_l c_l$. Choose $i,j,k$ so that $\pi^i(1)=m$, $\pi^j(2)=m$, and $\pi^k(3)=m$. We claim that $d_i$, $d_j$, and $d_k$ are linearly independent. Indeed, the first three entries of $d_i$, $d_j$, and $d_k$ have signs $(+,\ast_1,-),(-,+,\ast_2),(\ast_3,-,+)$ respectively, where $\ast_l$ is unknown. So, if $d_i$ is a multiple of $d_j$, then $\ast_1$ must be negative and $\ast_2$ must be positive. This then shows that $d_k$ is not a multiple of $d_j$.

Next, note that $d_i,d_j,d_k$ are contained in the two dimensional space $W=\{(w_1,\dots,w_n) : w_1+w_2+w_3=w_4=\dots=w_n=0\}$. So, $d_i,d_j,d_k$ span $W$ and so $W\subset H$. In particular, $(1,-1,0,\dots,0),(1,0,-1,\dots,0)\in H$ which implies $a_1=a_2=a_3$.
\end{proof}

As an application of Lemma \ref{lem:transpositions-3cycles}, we prove Theorem \ref{thm:conj-for-primes} for hyperplanes whose normal vector has a distinct entry.

\begin{corollary}
\label{cor:distinct-coord}
Let $p\geq3$ be a prime. Let $H=(a_1,\dots,a_p)^\perp$ be a hyperplane and assume there exists $i$ such that for all $j\neq i$ we have $a_j\neq a_i$. Then $|H\cap S_pv|\leq (p-1)!$ for all $v\in\RR^p$ with distinct coordinates.
\end{corollary}
\begin{proof}
After permuting coordinates, we may assume $i=1$. We will prove the Corollary by applying Lemma \ref{lem:n-1-not-contained}. Let $C$ be the elementary symmetric curve associated to $v$ and let $C_1,\dots,C_r$ be its irreducible components. For each $C_i$ in $H$, let $C_{ij}:=(1j)C_i$ where $j\neq 1$. Since $\sigma(C_i)\cap S_pv=\sigma(C_i\cap S_pv)$ for all $\sigma\in S_p$, the first part of Lemma \ref{lem:transpositions-3cycles} shows $H\cap C_{ij}\cap S_pv=\varnothing$. If $C_i$ and $C_k$ are contained in $H$, and if $C_{ij}=C_{kl}$, then $(1l)(1j)C_i=C_k$. If $j\neq l$, then $(1jl)C_i=C_k$; this is not possible by the second claim in Lemma \ref{lem:transpositions-3cycles}, where we take $\pi\in\stab(C_i)$ to be the $p$-cycle constructed in Lemma \ref{lem:stab-contain-p-cycle}. If $j=l$, then $C_i=C_k$ and so $i=k$. The result follows by Lemma \ref{lem:n-1-not-contained}.
\end{proof}

The rest of this section is concerned with the case where $H$ contains $v$ and $\sigma v$ for some 2-2-cycle $\sigma$. We start with the following preliminary result and as an application, prove Theorem \ref{thm:conj-for-primes} for special classes of hyperplanes.

\begin{lemma}
\label{lem:2-2-cycles-special}
Let $p\geq5$ be prime, $v\in\RR^p$ have distinct coordinates, and $C$ be the elementary symmetric curve associated to $v$ with some irreducible component $C_0$. Suppose $(ij)(kl)$ is a 2-2-cycle and $H=(a_1,\dots,a_p)^\perp$ is a hyperplane containing $C_0$, and $(ij)(kl) C_0$. If $a_i=a_k=a$ and $a_j=a_l=b$, then $a=b$.
\end{lemma}
\begin{proof}
By permuting coordinates, we can assume $v\in C_0$. From Corollary \ref{cor:inv-subspace}, we know $\Span_\CC{C_0}$ contains $\sigma w$ for some $\sigma\in S_p$, where $w=(\zeta,\dots,\zeta^p)$ and $\zeta=e^{2\pi i/p}$. Thus $H$ contains both $\sigma w$ and $(ij)(kl)\sigma w$. Subtracting we find $w-(ij)(kl)\sigma w\in H=(a_1,\dots,a_p)^\perp$. Set $\alpha'=\sigma^{-1}(\alpha)$, we have 
\[(a-b)(\zeta^{i'}-\zeta^{j'}+\zeta^{k'}-\zeta^{l'}) = a(\zeta^{i'}-\zeta^{j'})+b(\zeta^{j'}-\zeta^{i'})+a(\zeta^{k'}-\zeta^{l'})+b(\zeta^{l'}-\zeta^{k'})=0.\]
%\[(a-b)(\zeta^i-\zeta^j+\zeta^k-\zeta^l)=0.\]
Since $p\geq5$ and $i,j,k,l$ are distinct, we must have $a=b$.
\end{proof}

\begin{corollary}
\label{cor:product-max}
Let $p\geq5$ be a prime and $H=(a_1,\dots,a_p)^\perp$ be a hyperplane. Suppose $i_1,\dots,i_m$ are distinct, $j_1,\dots,j_n$ are distinct, $a_{i_1}=\dots=a_{i_m}$, $a_{j_1}=\dots=a_{j_n}$, and $a_{i_1}\neq a_{j_1}$. If $nm\geq p-1$, then $|H\cap S_pv|\leq (p-1)!$ for all $v\in\RR^p$ with distinct coordinates.
\end{corollary}
\begin{proof}
Let $C$ be the elementary symmetric curve associated to $v$. We will prove the corollary by applying Lemma \ref{lem:n-1-not-contained}. For each irreducible component $C_0$ of $C$ satisfying $C_0\subset H$, consider the $nm\geq p-1$ curves $$\{(i_k,j_l)C_0:1\leq k\leq m, 1\leq  l\leq n\}.$$ By Lemma \ref{lem:transpositions-3cycles}, we see $H\cap(i_k,j_l)C_0\cap S_pv=\varnothing$.

Next, if $C_0\subset H$ and $(i_k,j_l)C_0=(i_{k'},j_{l'})C_0$, then $H$ contains both $C_0$ and $(i_{k'},j_{l'})(i_k,j_l)C_0$. Similarly, if $H$ contains distinct irreducible component $C'_0$ and $C_0$, and if $(i_k,j_l)C_0=(i_{k'},j_{l'})C'_0$, then $H$ contains both $C_0$ and $(i_{k'},j_{l'})(i_k,j_l)C_0$. By Lemmas \ref{lem:transpositions-3cycles} and \ref{lem:2-2-cycles-special}, this is not possible as $(i_{k'},j_{l'})(i_k,j_l)$ is either a 2-2-cycle or a 3-cycle; in the case of a 3-cycle, we apply Lemma \ref{lem:transpositions-3cycles} by taking $\pi\in\stab(C_0)$ to be the $p$-cycle constructed in Lemma \ref{lem:stab-contain-p-cycle}.
\end{proof}

As a further application, we prove Theorem \ref{thm:conj-for-primes} for $p=3,5$.

\begin{corollary}
\label{cor:p=3,5}
Let $p\in\{3,5\}$ and $H=(a_1,\dots,a_p)^\perp$ be a hyperplane of $\RR^p$. If $v\in\RR^p$ has distinct coordinates not summing to $0$, then $|H\cap S_pv|\leq (p-1)!$.
\end{corollary}
\begin{proof}
Since the coordinates of $v$ do not sum to $0$, we know $H\neq(1,\dots,1)^\perp$. For $p=3$, our desired result then follows directly from Corollary \ref{cor:distinct-coord}. When $p=5$, %by Corollary \ref{cor:distinct-coord} all coordinates of $H^\perp$ are repeated at least once. Without loss of generality say 
Corollary \ref{cor:distinct-coord} reduces us to the case $H=(a,a,b,b,b)^\perp$ for some distinct $a,b\in \RR$. Our result then follows from Corollary \ref{cor:product-max}.
\end{proof}

We end this section with some more refined results concerning the structure of hyperplanes that contain $v$ and $\sigma v$ with $\sigma$ a 2-2-cycle.

\begin{lemma}
\label{lem:2-2-cycles}
Let $p\geq 5$ be prime, $v\in\RR^p$ have distinct coordinates, and $H=(a_1,\dots,a_p)^\perp$ be a hyperplane. Suppose $\sigma=(ij)(kl)$ is a 2-2-cycle and $\pi$ is an $p$-cycle. Let $G\subset S_p$ be a subgroup that contains $\pi$ and assume $\dim\Span(Gv)>3$. If $H$ contains $Gv$ and $\sigma G v$, then $a_i=a_j$ and $a_k=a_l$.
\end{lemma}
\begin{proof}
By permuting coordinates we can assume $\pi=(12\dots p)$. Note that the subspace $\Span(Gv)\subset\RR^p$ is invariant under the action of $\langle\pi\rangle\simeq\ZZ/p$. Since $\dim\Span(Gv)>3$, when viewed as a complex $\ZZ/p$-representation, it contains $w_1=(1,\zeta,\dots,\zeta^{p-1})$, $w_2=(1,\zeta^{-1},\dots,\zeta^{-(p-1)})$, and $w_3=(1,\zeta^{m},\dots,\zeta^{m(p-1)})$ for some primitive $p$-th root of unity $\zeta$ and some $m\neq0,\pm1$ mod $p$.

For $d\in\{1,2,3\}$, let $t_d$ be such that the $i$-th coordinate of $u_d:=\zeta^{t_d}w_d$ is $1$, where $i$ is as in the statement of the lemma. Since $u_d\in\<\pi\>w_d\subset\Span(Gv)$, we see $\sigma u_d\in\Span(\sigma Gv)$, so in particular,
\[
\sigma u_d - u_d\in H.
\]
Let $x$ be the $j$-th coordinate of $u_1$. Then the $k$-th and $l$-th coordinates of $u_1$ are, respectively, $x^a$ and $x^b$ for distinct $a,b\in\{2,\dots,p-1\}$. The $j$-th, $k$-th, and $l$-th coordinates of $u_2$ are then $x^{-1}$, $x^{-a}$, and $x^{-b}$, respectively. The $j$-th, $k$-th, and $l$-th coordinates of $u_3$ are $x^{m}$, $x^{ma}$, and $x^{mb}$, respectively. Letting $\alpha=1-x$, $\beta=x^a-x^b$, $\alpha'=1-x^{-1}$, and $\beta'=x^{-a}-x^{-b}$, we find
\[
(\dots,\alpha,\dots,-\alpha,\dots,\beta,\dots,-\beta,\dots)=\sigma u_1-u_1\in H
\]
and 
\[
(\dots,\alpha',\dots,-\alpha',\dots,\beta',\dots,-\beta',\dots)=\sigma u_2-u_2\in H
\]
where the omitted entries are 0, and the four non-zero entries are in the $j,i,l,k$-th positions respectively. 

If $(\alpha,\beta),(\alpha',\beta')$ are linearly independent, then have $(\dots,1,\dots,-1,\dots,0,\dots,0,\dots)\in H$ which implies $a_i=a_j$, and consequently $a_k=a_l$.

Next suppose $(\alpha,\beta)$ and $(\alpha',\beta')$ are linearly dependent. Then %$\frac{\alpha}{\beta}=\frac{\alpha'}{\beta'}$, so 
\[\frac{x^a - x^b}{1-x}  = \frac{\beta}{\alpha} = \frac{\beta'}{\alpha'} = \frac{x^{-a} - x^{-b}}{1-x^{-1}} = \frac{x^{1-a} - x^{1-b}}{x-1}\]
and hence
\begin{equation}
\label{poly1}
x^a - x^b = x^{1-b} - x^{1-a}.
\end{equation}
Since $b\neq a$, this is a contradiction by the linear independence of roots of unity over $\QQ$, unless $a+b=1$ mod $p$. 

So, we may suppose $a+b=1$ mod $p$. Consider
\[
(\dots,\alpha'',\dots,-\alpha'',\dots,\beta'',\dots,-\beta'',\dots)=\sigma u_3-u_3\in H
\]
where $\alpha''=1-x^m$ and $\beta''=x^{ma}-x^{mb}$. If $(\alpha,\beta),(\alpha'',\beta'')$ are linearly independent, we again arrive at our desired conclusion that $a_i=a_j$ and $a_k=a_l$, so we may assume $(\alpha,\beta),(\alpha'',\beta'')$ are linearly dependent. Then
\[\frac{1-x}{x^{a}-x^{1-a}}=\frac{\alpha}{\beta}=\frac{\alpha''}{\beta''}=\frac{1-x^{m}}{x^{ma}-x^{m(1-a)}}\]
and so
\begin{equation}
\label{poly2}
x^{ma}-x^{m(1-a)}-x^{ma+1}+x^{m-ma+1}-x^a+x^{1-a}+x^{m+a}-x^{1-a+m}=0.
\end{equation}
Let $f(x)$ be the polynomial (\ref{poly2}), where the exponents are taken to be numbers between $0$ and $p$ by reducing mod $p$, and we now view $x$ as an indeterminate. Since $f(x)$ has integer coefficients, has degree less than $p$, and has a primitive $p$-th root of unity as a root, it is a constant multiple of the $p$-th cyclotomic polynomial. Note the term $x^{ma}$ cannot be cancelled by any other term since $m\neq0,\pm1$ mod $p$, so $f(x)$ is a non-zero polynomial with at most 8 terms. In particular, it is not a multiple of the cyclotomic polynomial for $p\geq11$.

When $p= 5,7$, since $f(x)$ is a non-zero constant multiple of the $p$-th cyclotomic polynomial, some of the terms in $f(x)$ must cancel to yield exactly $p$ terms all with the same non-zero coefficient. This is impossible, however, as $f(x)$ has 4 terms with coefficient equal to $1$ and 4 terms with coefficient equal to $-1$.
\end{proof}

%\matt{Small comment: in this next lemma as well as Section 6, I'm not a huge fan of the $j'$ notation since, e.g.~$5'$ is weird. But I also understand that superscripts with $\sigma(j)$ can be cumbersome. Any other notation we can use? If you're both fine with $j'$, then I'm ok leaving it too.}

\begin{lemma}
\label{lem:i-j-k-l-identity}
Let $p\geq7$ be a prime, $\zeta=e^{2\pi i/p}$, $w=(\zeta,\zeta^2,\dots,\zeta^p)$, and $\sigma\in S_p$. If $a_1,\dots,a_p\in\RR$ and $H=(a_1,\dots,a_p)^\perp$ is a hyperplane that contains both $\sigma w$ and $(ij)(kl)\sigma w$ with $i,j,k,l$ distinct, then either
\[\sigma^{-1}(i)+\sigma^{-1}(j)=\sigma^{-1}(k)+\sigma^{-1}(l)\mod p\]
or $a_i=a_j$ and $a_k=a_l$.
\end{lemma}
\begin{proof}
For ease of notation, we let $\alpha'=\sigma^{-1}(\alpha)$ for $\alpha=1,\dots,p$. First note that $H$ contains the element
\[\sigma w-(ij)(kl)\sigma w=(\dots,\zeta^{j'}-\zeta^{i'},\dots,\zeta^{i'}-\zeta^{j'},\dots,\zeta^{l'}-\zeta^{k'},\dots,\zeta^{k'}-\zeta^{l'},\dots)\]
where the omitted entries are 0, and the four non-zero entries are in the $j,i,l,k$-th positions respectively. Since $(a_1,\dots,a_p)$ is a real vector, $H$ also contains the complex conjugate vector %$\sigma \overline{w}-(ij)(kl)\sigma \overline{w}$, which is
\[
%\sigma \overline{w}-(ij)(kl)\sigma \overline{w}=
(\dots,\zeta^{-j'}-\zeta^{-i'},\dots,\zeta^{-i'}-\zeta^{-j'},\dots,\zeta^{-l'}-\zeta^{-k'},\dots,\zeta^{-k'}-\zeta^{-l'},\dots).
\]

Now, if the two vectors $(\zeta^{j'}-\zeta^{i'},\zeta^{l'}-\zeta^{k'}),(\zeta^{-j'}-\zeta^{-i'},\zeta^{-l'}-\zeta^{-k'})$ are linearly independent, then $(\dots,1,\dots,-1,\dots,0,\dots,0,\dots)\in H$, which means $a_i=a_j$, from which it follows that $a_k=a_l$. Otherwise,
\[\frac{\zeta^{j'}-\zeta^{i'}}{\zeta^{l'}-\zeta^{k'}}=\frac{\zeta^{-j'}-\zeta^{-i'}}{\zeta^{-l'}-\zeta^{-k'}}\]
and hence 
\[-\zeta^{j' - k'} + \zeta^{j' - l'} + \zeta^{-k' + i'} - \zeta^{-l' + i'}  +\zeta^{k' - j'} - \zeta^{l' - j'} - \zeta^{k' - i'} + \zeta^{l' - i'}=0.\]

Consider the polynomial $$f(z):=-z^{j' - k'} + z^{j' - l'} + z^{-k' + i'} - z^{-l' + i'}  +z^{k' - j'} - z^{l' - j'} - z^{k' - i'} + z^{l' - i'}$$ where we view the exponents as numbers between $0$ and $p$ by reducing mod $p$. Since $\deg f(z)<p$ and since $f(z)$ is a polynomial with integer coefficients satisfying $f(\zeta)=0$, it must be the case that $f(z)$ is a constant multiple of the $p$-th cyclotomic polynomial. For $p\geq11$, since $f(z)$ has at most $8$ terms, this forces $f(z)=0$; in particular two terms of $f(z)$ must cancel. Similarly, for $p=7$, we know $f(z)$ has at most $6$ terms, and so two terms in the above expression must cancel. In all cases, when $p\geq7$, we must have $z^{j'-k'}=z^{l'-i'}$ since $i,j,k,l$ are distinct mod $p$.
%Since $p\geq7$ and $\deg f<p$, we must have $f(z)=0$. It follows that $z^{j'-k'}=z^{l'-i'}$ as $i,j,k,l$ are distinct mod $p$. 
So, $j'-k'=l'-i'$ mod $p$, and hence $i'+j'=k'+l'$ mod $p$.
\end{proof}

\section{Theorem \ref{thm:conj-for-primes} in the non-dihedral case}
\label{sec:non-D2p}

Given the algebro-geometric results in Section \ref{sec:AG-input}, the proof of Theorem \ref{thm:conj-for-primes} is divided into two cases, depending on whether or not the stabilizer of $C_1$ is the dihedral group $D_{2p}$ with $2p$ elements. In this section, we prove the following result, which handles the non-dihedral case:

\begin{theorem}
\label{thm:conj-for-primes-special}
Let $p$ be a prime and $v\in\RR^p$ have distinct coordinates that do not sum to 0. If $\stab(C_1)\not\simeq D_{2p}$, then $\max_H |H\cap S_nv|=(p-1)!$.
\end{theorem}

Given a subgroup $G'$ of $G$, we let $N_G(G')$ denote the normalizer of $G'$ in $G$. We recall the following two theorems, which we use to obtain a structure result for $\stab(C_1)$.
%the stabilizers of the curves constructed in Section \ref{sec:AG-input}.

\begin{theorem}[{Burnside, \cite{burnsides-thm}}]
\label{thm:burnside}
%\kent{http://people.uleth.ca/~dave.morris/papers/Permp2.pdf}\\
For $p$ prime, a transitive subgroup of $S_p$ is either doubly transitive or contains a normal Sylow $p$-subgroup.
\end{theorem}

\begin{theorem}[{\cite[Exercise 2.6]{serre-rept-thy}}]
\label{thm:doubly-trans} 
%\kent{linear Representations of finite groups -Serre Pg17}\\
If $G$ is a doubly transitive subgroup of $S_n$, then the permutation representation $\RR^n$ is the direct sum of two irreducible $G$-representations: the trivial representation and the standard representation of $S_n$.
\end{theorem}

\begin{proposition}
\label{prop:stabilizer-special-form}
Let $p$ be prime, $v\in\RR^p$ have distinct coordinates that do not sum to 0. Suppose $v\in C_0\subset H$ where $H$ is a hyperplane of $\RR^p$ and $C_0$ is an irreducible component of the elementary symmetric curve associated to $v$. Then
\[
\stab(C_0)=\<\pi,\sigma\>\subset N_{S_p}(\langle\pi\rangle)
\]
where $\pi$ is a $p$-cycle, and $\sigma$ is a power of some $(p-1)$-cycle.
\end{proposition}
\begin{proof}
To ease notation, let $G=\stab(C_0)$. By Lemma \ref{lem:stab-contain-p-cycle}, $G$ contains a $p$-cycle $\pi$ and hence is a transitive subgroup of $S_p$. Note that $\<\pi\>$ is a Sylow $p$-subgroup of $G$.

Our first goal is to show $G\subset N_{S_p}(\langle\pi\rangle)$. If this is not the case, then $\langle\pi\rangle$ is not normal in $G$, and so $G$ is doubly transitive by Theorem \ref{thm:burnside}. Notice that 
\[
(1,\dots,1)=\frac{1}{\sum_i v_i}\sum_{i=0}^{p-1} \pi^iv\in \Span(\langle\pi\rangle v)\subset \Span(Gv)
\]
so $\Span(Gv)$ contains the trivial representation. Since $v\in\Span(Gv)$ and $v$ has distinct coordinates, we see $\Span(Gv)$ cannot equal the trivial representation. It follows then from Theorem \ref{thm:doubly-trans} that $\Span(Gv)=\RR^p$. On the other hand,
\[
\RR^p=\Span(Gv)\subset \Span(GC_0)=\Span(C_0)\subset H
\]
which contradicts the fact that $H$ is a hyperplane. We have therefore proven our claim that $G\subset N_{S_p}(\langle\pi\rangle)$.

%https://www.math.umd.edu/~tjh/600_fall07_final_solutions.pdf
Next, one readily checks that $N_{S_p}(\langle\pi\rangle)=\<\pi,\tau\>$ where $\tau$ is a $(p-1)$-cycle such that $\tau^{-1}\pi\tau=\pi^k$ with $k$ a generator for $(\ZZ/p)^*$. In particular, $N_{S_p}(\langle\pi\rangle)\simeq\<\pi\>\rtimes(\ZZ/p)^*$ where $(\ZZ/p)^*$ acts on $\<\pi\>\simeq\ZZ/p$ in the natural way. Since $G$ is a subgroup of $N_{S_p}(\langle\pi\rangle)$ that contains $\pi$, we see $G=\<\pi\>\rtimes Q$, where $Q$ is a subgroup of $(\ZZ/p)^*$. It follows that $G=\<\pi,\sigma\>$ where $\sigma=\tau^i$ for some $i$.

\end{proof}

Given the above structure result for $\stab(C_0)$, we next understand how $\CC^p$ decomposes as a $\stab(C_0)$-representation.

\begin{lemma}
\label{lem:dim-at-least-5}
Let $p$ be prime, $\pi\in S_p$ be a $p$-cycle, and $G$ be a subgroup of $N_{S_p}(\langle\pi\rangle)$. Then every non-trivial complex irreducible $G$-subrepresentation of the permutation representation $\CC^p$ has dimension $|G/\<\pi\>|$.
%Let $G:=\langle\pi,\sigma\rangle$ be a subgroup of $N(\pi)$, then any nontrivial complex representation of $G<S_p$ in $\CC^p$ has dimension at least $|\sigma|:=\order(\sigma)$.
\end{lemma}
\begin{proof} 
As in the proof of Proposition \ref{prop:stabilizer-special-form}, we know $G=\langle\pi,\sigma\rangle$ where $\sigma^{-1}\pi\sigma=\pi^k$. Fix a primitive $p$-th root of unity $\zeta$. Decomposing $\CC^p$ into irreducible subrepresentations of $\<\pi\>\simeq\ZZ/p$, we have $\CC^p=\bigoplus_{i\in\ZZ/p} V_i$ where %$\omega_i=(1,\zeta^i,\zeta^{2i},\dots,\zeta^{(p-1)i})$ and 
$V_i=\Span(\omega_i)$ and $\pi\omega_i=\zeta^i\omega_i$. We find $\pi\sigma\omega_i=\sigma\pi^k\omega_i=\zeta^{ik}\sigma\omega_i$ and hence $\sigma V_i=V_{ik}$. So, the non-trivial irreducible $G$-subrepresentations of $\CC^p$ are given by $\Span(GV_i)=\bigoplus_j V_{ik^j}$, where the sum runs over $0\leq j<\ord(k)$ and $\ord(k)$ is the order of $k$ in $(\ZZ/p)^*$, i.e.~the order of $G/\<\pi\>$.
\end{proof}

\begin{proof}[{Proof of Theorem \ref{thm:conj-for-primes-special}}]
By Corollary \ref{cor:p=3,5}, we may assume $p\geq 7$. Let $C$ be elementary symmetric curve associated to $v$ and let $C_1,\dots,C_r$ be its irreducible components. By Lemma \ref{lem:transitive-action-components}, we may assume that $H$ contains an irreducible component of $C$; without loss of generality, $v\in C_1\subset H$. Letting $G=\stab(C_1)$, we know from Proposition \ref{prop:stabilizer-special-form} that $G=\<\pi,\sigma\>$ where $\pi$ is a $p$-cycle and $\sigma^{-1}\pi\sigma=\pi^k$. Since $G\not\cong D_{2p}$, the order of $\sigma$, $\ord(\sigma)$, cannot equal $2$.
%Note that $r=p!/|G|$.%=(p-1)!/|G/\<\pi\>|$.

Next, note that $\Span(\langle\pi\rangle v)$ contains the trivial representation, as $(1,\dots,1)=\frac{1}{\sum_i v_i}\sum_{i=0}^{p-1} \pi^iv$. On the other hand, $\Span(\langle\pi\rangle v)$ cannot equal the trivial representation since it contains $v$, which has distinct coordinates. So, $\Span(Gv)$ contains both the trivial and a non-trivial $G$-subrepresentation of $\RR^p$.

If $\sigma=1$, then $G=\<\pi\>$ and since $r=p!/|G|=(p-1)!$, we see from Lemma \ref{lem:irred-comp-same-degree} that $\deg(C_1)=1$, i.e.~the curve $C_1$ is a line. %As in the proof of Proposition \ref{prop:stabilizer-special-form}, we see 
Since the non-trivial irreducible $\<\pi\>$-subrepresentations of $\RR^p$ are all $2$-dimensional, it follows that $\dim\Span(\langle\pi\rangle v)\geq3$. In particular, the line $C_1$ cannot contain $\<\pi\>v$.

So, we may assume $\ord(\sigma)\geq3$. Let $H=(a_1,\dots,a_p)^\perp$. Again by Lemma \ref{lem:transpositions-3cycles}, $(ij)H=H$ if and only if $a_i=a_j$. Since the coordinates of $v$ do not sum to $0$, not all of the $a_i$ are equal. From this, it is straightforward to check that there are at least $p-1$ distinct transpositions $\tau_1,\dots,\tau_{p-1}$ such that $\tau_k H\neq H$. 
For each $C_i$ contained in $H$, let $C_{ij}=\tau_jC_i$. We will check that the conditions in Lemma \ref{lem:n-1-not-contained} are satisfied, and conclude $|H\cap S_pv|\leq (p-1)!$.

It follows directly from Lemma \ref{lem:transpositions-3cycles} that $H\cap C_{ij}\cap S_pv=\varnothing$. Next, suppose $H$ contains $C_i$ and $C_j$, and that we have $C_{ik}= C_{jl}$. If $k\neq l$, then $H$ contains both $C_j$ and $C_i=\tau_k\tau_lC_j$. Now, $\tau_k\tau_l$ cannot be a 3-cycle as this would contradict Lemma \ref{lem:transpositions-3cycles}. So, $\tau_k\tau_l$ must be a 2-2-cycle, 
%However, $\tau_k\tau_l$ is either a 3-cycle, which contradicts Lemma \ref{lem:transpositions-3cycles}, or a 2-2-cycle, 
in which case we note that $\Span(Gv)$ contains the trivial and a non-trivial irreducible $G$-subrepresentation of $\RR^p$. So $\dim\Span(Gv)\geq \ord(\sigma)+1>3$ by Lemma \ref{lem:dim-at-least-5}, which gives a contradiction by Lemma \ref{lem:2-2-cycles}. It follows that $k=l$, so $C_i=C_j$ and $i=j$.
%\kent{ we have refered to Lemma \ref{lem:transpositions-3cycles}  3 times in the last 2 paragraphs. maybe we should split the it apart like \ref{l:basic-C_a-facts}}
\end{proof}

\section{Completing the proof of Theorem \ref{thm:conj-for-primes}}
\label{sec:completing-pf}

To finish the proof of Theorem \ref{thm:conj-for-primes}, we must now handle the case not covered by Theorem \ref{thm:conj-for-primes-special} and Corollary \ref{cor:p=3,5}, namely when $p\geq 7$ and the irreducible components of $C$ have stabilizers isomorphic to $D_{2p}$, the dihedral group with $2p$ elements. Note that by Lemma \ref{lem:irred-comp-same-degree}, this implies the irreducible components of $C$ have degree 2.

%.=\<\pi,\sigma\>$ where $\pi$ is a $p$-cycle and $\sigma^2=1$. As a result of Proposition \ref{prop:stabilizer-special-form}, the $2p$ elements of the stabilizer are composed of $p$ elements in form of $p$-cycle, and $p$ elements in form of product of $(p-1)/2$ disjoint transpositions. Also, instead of having the permutation $S_p$ to permute the set $\{1,\dots,p\}$, we let it permute the set $\{0,\dots,p-1\}$ and change our cycle notation for permutations. This way when $\sigma$ acts on the vector $(1,\zeta,\dots,\zeta^{p-1})$, we have $\sigma(1,\zeta,\dots,\zeta^{p-1})=(\zeta^{\sigma(0)},\zeta^{\sigma(1)},\dots,\zeta^{\sigma(p-1)})$.

Let $H=(b_1,\dots,b_p)^\perp$ be any hyperplane. We fix the following notation. Let $\{1,\dots,p\}=\lambda_1\cup\dots\cup\lambda_K$ be the partition defined by the domains on which the function $j\mapsto b_j$ are constant. In other words, we have distinct $a_1,\dots,a_K\in\RR$ such that $b_j=a_J$ if and only if $j\in\lambda_J$. Let
\[
m:=\min_J |\lambda_J| \quad\textrm{and}\quad |\lambda_M|=m
\]
for some fixed choice of $M$.
By Corollaries \ref{cor:distinct-coord} and \ref{cor:product-max}, we can assume
\[
2\leq m<\sqrt{p-1}.
\]
If $a_M\neq 0$, we may scale to assume $a_M=1$.

We prove Theorem \ref{thm:conj-for-primes} by studying properties of a graph $\Gamma$ which we now define. Throughout the rest of Section \ref{sec:completing-pf}, we fix two distinct elements $i,k\in\lambda_M$ and let
 \[
 T:=\{(ij):j\notin \lambda_M\}\cup\{(kj):j\notin \lambda_M\}.
 \]
Let $\Gamma:=\Gamma_{ik}$ be the graph whose vertices and edges are defined as follows. Let $C$ be the elementary symmetric curve associated to $v$. The vertices of $\Gamma$ are the irreducible components $C_a$ of $C$ for which $C_a\subset H$. Let $|\Gamma|$ denote 
%\[t:=|\Gamma|\]be 
the number of vertices in $\Gamma$. If $C_1,C_2\in\Gamma$ are two vertices, we write $C_1\sim C_2$ when $C_1$ and $C_2$ are connected by an edge. The edges of $\Gamma$ are defined by
\[
C_1\sim C_2 \quad\Longleftrightarrow\quad (ij)(kl)C_1=C_2 \textrm{\ for\ distinct\ } j,l\notin \lambda_M;
\]
here $i,k$ are the elements that we have fixed above. %(the definition is symmetric since $(ij)(kl)C_1=C_2$ if and only if $(ij)(kl)C_2=C_1$).

We observe that for each $C_0\in\Gamma$, if $\sigma,\tau\in T$ and $\sigma C_0=\tau C_0$, then $\sigma^{-1}\tau\in \stab(C_0)\simeq D_{2p}$. Since $\sigma^{-1}\tau$ is a product of two transpositions, it is not a $p$-cycle nor is it a product of $(p-1)/2$ disjoint 2-cycles, so $\sigma^{-1}\tau=1$. Thus, %viewing $\{\sigma C_0:\sigma\in T\}$ as a subset of the set of irreducible components of $C$, 
we find
\[
|\{\sigma C_0:\sigma\in T\}|=|T|=2(p-m).
\]

For the rest of the section, we let $$w=(\zeta,\dots,\zeta^p),$$ where $\zeta=e^{2\pi i/p}$.

\begin{lemma}
\label{l:unique-edge-labeling-by-j-l}
If $C_1,C_2\in\Gamma$ and $C_1\sim C_2$, then $(ij)(kl)C_1=C_2$ for a unique pair $(j,l)$.
\end{lemma}
\begin{proof}
Suppose $(ij)(kl)C_1=(ij')(kl')C_1=C_2\subset H$, where $j\neq j'$ or $l\neq l'$. Then 
\[(ij)(kl)(ij')(kl')\in\stab(C_1)\simeq D_{2p}.\]
By Corollary \ref{cor:inv-subspace}, $\Span_\CC C_1$ contains $\sigma w$ for some $\sigma\in S_p$. Then $H$ contains $\sigma w$, $(ij)(kl)\sigma w$, and $(ij')(kl')\sigma w$, so by Lemma \ref{lem:i-j-k-l-identity}, the following two equations hold:
\[\sigma^{-1}(i)+\sigma^{-1}(j)=\sigma^{-1}(k)+\sigma^{-1}(l)\mod p\]
\[\sigma^{-1}(i)+\sigma^{-1}(j')=\sigma^{-1}(k)+\sigma^{-1}(l')\mod p.\]
Subtracting the equations, we find
\[\sigma^{-1}(j')-\sigma^{-1}(j)=\sigma^{-1}(l')-\sigma^{-1}(l)\mod p.\]
This implies that $j=j'$ if and only if $l=l'$, and hence $j\neq j'$. In addition $j'=l'$ implies $j=l \mod p$ and $i=k \mod p$, which is also not true. Recall that neither $i$ nor $j'$ is equal to $k$ or $l$ mod $p$.

Putting these observations together we see that $(ij)(kl)(ij')(kl')$ is not the identity, as it sends $j'$ to $j$. We see that $(ij)(kl)(ij')(kl')$ also does not permute $p\geq 7$ elements. So as an element of $D_{2p}$, it must be a product of $(p-1)/2$ disjoint transpositions, which is only possible when $p=7$ and $(p-1)/2=3$. However, $(ij)(kl)(ij')(kl')$ is an even permutation so this is also not possible when $p=7$. We have thus established our claim.
\end{proof}

Finally, we let
\[
\mathcal{T}=\{\sigma C_0:\sigma\in T,C_0\in\Gamma\}.
%=\bigcup_{C_0\in\Gamma}TC_0.
\]
%viewed as a subset of the irreducible components of $C$.
Note that $\Gamma\cap\mathcal{T}=\varnothing$ by Lemma \ref{lem:transpositions-3cycles}.

We divide the proof of Theorem \ref{thm:conj-for-primes} into two cases depending on the size of $\mathcal{T}$. The following result easily dispenses with the case where $\mathcal{T}$ is big.

\begin{lemma}
\label{l:bound-size-of-mathcalT}
If $|\mathcal{T}|\geq (p-1)|\Gamma|$, then $|H\cap S_pv|\leq(p-1)!$. 
\end{lemma}
\begin{proof}
If $D$ is an irreducible component of $C$ and $D\subset H$, then $|H\cap D\cap S_pv|=\frac{p!}{r}$. If $D\not\subset H$, then by B\'ezout's Theorem and Lemma \ref{lem:irred-comp-same-degree}, we have $|H\cap D\cap S_pv|\leq\frac{(p-1)!}{r}$. Furthermore, if $D\in \mathcal{T}$, then by Lemma \ref{lem:transpositions-3cycles}, we have $H\cap D\cap S_pv=\varnothing$. Putting these bounds together, and making use of the fact that $\Gamma\cap\mathcal{T}=\varnothing$, we find
\begin{equation*}
%\label{eq:counting}
\begin{split}
 |H\cap S_pv| &\leq\sum_{D\subset H} \frac{p!}{r}+\sum_{\substack{D\not\subset H\\ D\notin \mathcal{T}}}\frac{(p-1)!}{r}+\sum_{\substack{D\not\subset H\\ D\in\mathcal{T}}} 0\\
&\leq \frac{p!}{r}|\Gamma|+(r-|\Gamma|-(p-1)|\Gamma|)\frac{(p-1)!}r\\
&=(p-1)!.\qedhere
\end{split}
\end{equation*}
\end{proof}

The goal of the rest of Section \ref{sec:completing-pf} is to prove that $|\mathcal{T}|\geq(p-1)|\Gamma|$, and hence Theorem \ref{thm:conj-for-primes} holds in light of Lemma \ref{l:bound-size-of-mathcalT}. To this end, we assume throughout the rest of Section \ref{sec:completing-pf} that
\[
|\mathcal{T}|<(p-1)|\Gamma|
\]
and aim to arrive at a contradiction.

\begin{lemma}
\label{lem:deg-bound-vertices}
%Suppose $p\geq7$.
If $D\in\Gamma$ is a vertex, let $d(D)$ be its degree in $\Gamma$. Then there exists $C_0\in\Gamma$ such that 
\[
\sum_{D\sim C_0}d(D)\geq \kappa+2d(C_0),
\]
where $\kappa=(p-2m)^2-1$.
\end{lemma}
\begin{proof}
We begin by counting the number of elements in $\T$. Note that if $C_1,C_2\in\Gamma$ are distinct, then we cannot have $(ij)C_1=(il)C_2$ since this would imply $j\neq l$ and $C_2=(ijl)C_1$, contradicting Lemma \ref{lem:transpositions-3cycles}. Next notice that if $(ij)C_1=(kl)C_2$, then $(kl)(ij)C_1=C_2$ and so Lemma \ref{lem:transpositions-3cycles} shows we must have $j\neq l$, i.e.~$C_1\sim C_2$. Conversely, if $C_1\sim C_2$, then we have already established that there is a unique pair $(j,l)$ for which $(ij)(kl)C_1=C_2$; it follows that $(ij)C_1=(kl)C_2$ and $(kl)C_1=(ij)C_2$. Putting these observations together, we see that if $e$ is the number of edges of $\Gamma$, then
%Putting these observations together, we see that for each edge $C_1\sim C_2$ corresponding to $(ij)(kl)C_1=C_2$, we have $(kl)C_1\in TC_1\cap TC_2$ and $(kl)C_2\in TC_1\cap TC_2$; moreover, every element of $TC_1\cap TC_2$ yields such an edge. We have therefore shown
\[
|\T|=|T||\Gamma|-2e=2(p-m)|\Gamma|-2e.
\]
%Let us count the number of elements in $\mathcal{T}$. Note that if $(ij)C_1=(kl)C_2$, then $(kl)(ij)C_1=C_2$, so by Lemma \ref{lem:transpositions-3cycles}, we must have $j\neq l$, and hence $C_1\sim C_2$. Conversely, if two vertices are connected by an edge, say $(ij)(kl)C_1=C_2$, then we have already established that the pair $(j,l)$ is uniquely determined, and we see $(ij)C_1=(kl)C_2$.  , so the element $(kl)C_2$ shows up in both $TC_1$ and $TC_2$, and when we count it in $TC_2$ as $(kl)C_2$, it is double counted with $(ij)C_1$ in $TC_1$. Similarly since $(kl)C_1=(ij)C_2$, we counted $(kl)C_1$ twice. Let $t$ be the number of vertices in the graph, then we need to subtract two from the total, $2(p-m)|\Gamma|$ whenever such double counting occurs. That is, if there are $e$ pairs of curves connected by edges, then we subtract $2e$ from the total $2(p-m)|\Gamma|$. 
Since $|\mathcal{T}|<(p-1)|\Gamma|$, we have 
\[
e>\frac{(p-2m+1)|\Gamma|}{2}.
%e>\frac{2(p-m)|\Gamma|-(p-1)|\Gamma|}{2}=\frac{(p-2m+1)|\Gamma|}{2}.
\]

Suppose that $\sum_{D\sim C}d(D)<\kappa+2d(C)$ for all vertices $C\in\Gamma$. 
%for some number $A$, which we will choose later to get a contradiction.
Then we see %since the sum of degrees of the vertices of $\Gamma$ is twice the number of edges, we have
\[\sum_{C\in\Gamma}\sum_{D\sim C}d(D)<\sum_{C\in\Gamma}(\kappa+2d(C))=\kappa|\Gamma|+2\sum_{C\in\Gamma}d(C)=\kappa|\Gamma|+4e.\]
One readily checks that
\[\sum_{C\in\Gamma}\sum_{D\sim C}d(D)=\sum_{C\in\Gamma}d(C)^2.\]
By the Cauchy--Schwartz inequality, we see
%since the sum of degree of all vertices is $2e$, the average degree is $2e/t$ and by Cauchy-Schwartz,
\[\sum_{C\in\Gamma}d(C)^2\geq |\Gamma|\left(\frac{1}{|\Gamma|}\sum_{C\in\Gamma}d(C)\right)^2=\frac{4e^2}{|\Gamma|}.\]
Thus, $\kappa|\Gamma|+4e>4e^2/|\Gamma|$ and so %$\kappa|\Gamma|^2>4e(e-|\Gamma|)$.
%
%Since $p\geq 7$ and $m<\sqrt{p-1}$, we see $e>\frac{1}{2}|\Gamma|(p-2m+1)>|\Gamma|>0$ and so
\[
\begin{split}
\kappa|\Gamma|^2>4e(e-|\Gamma|) &> 4|\Gamma|\frac{p-2m+1}2\left(|\Gamma|\frac{(p-2m+1)}2-|\Gamma|\right)\\
&=|\Gamma|^2(p-2m+1)(p-2m-1)=\kappa|\Gamma|^2,
\end{split}
\]
a contradiction.
\end{proof}

Throughout the rest of this section, we fix $C_0$, $\kappa$, and $d:=d(C_0)$ as in Lemma \ref{lem:deg-bound-vertices}. We prove

\begin{proposition}\label{prop:projection-injective-new}
$\sum_{D\sim C_0}d(D)\leq p-m+2d$.
\end{proposition}

Assuming Proposition \ref{prop:projection-injective-new} for the moment, let us complete the proof of Theorem \ref{thm:conj-for-primes}. By Lemma \ref{lem:deg-bound-vertices} and Proposition \ref{prop:projection-injective-new}, we have $\kappa\leq \sum_{D\sim C_0}d(D)-2d\leq p-m$. Now, if $p=7$, then $2\leq m<\sqrt{p-1}$ implies $m=2$ and hence $\kappa=(7-4)^2-1=8>7-2=p-m$, a contradiction. If $p>7$, then
\[
\kappa=(p-2m)^2-1>(p-2\sqrt p)^2-1>p-2\geq p-m,
\]
again a contradiction.

\vspace{2em}

The rest of Section \ref{sec:completing-pf} is devoted to the proof of Proposition \ref{prop:projection-injective-new}. The proof is based on an analysis of the edges in the second-order neighborhood of $C_0$. By definition of $C_0$, it has $d$ neighbors $C_1,\dots,C_d$ such that the sum of the degrees of these neighbors is at least $\kappa+2d$. We have $j_1,\dots,j_d,l_1,\dots,l_d\notin \lambda_M$ with $j_a\neq l_a$ such that
\[
C_a:=(ij_a)(kl_a)C_0.
\]
For notational convenience, let $j_0=k$ and $l_0=i$ so that $C_0=(ij_0)(kl_0)C_0$.

\begin{lemma}
\label{l:basic-C_a-facts}
We have the following:
\begin{enumerate}
\item\label{l:basic-C_a-facts::jaia-equation} For $0\leq a\leq d$,
\[
\sigma^{-1}(i)+\sigma^{-1}(j_a)=\sigma^{-1}(k)+\sigma^{-1}(l_a)\mod p.
\]
In particular, $j_a$ determines $l_a$, and $l_a$ determines $j_a$.
\item\label{l:basic-C_a-facts::distinctness} 
%\[
$j_0,\dots,j_d \textrm{\ are\ distinct\ and\ } l_0,\dots,l_d \textrm{\ are\ distinct}$.
%\]
\end{enumerate}
\end{lemma}
\begin{proof}
From Corollary \ref{cor:inv-subspace}, there exists $\sigma\in S_p$ such that $\sigma w=(\zeta^{\sigma^{-1}(1)},\dots,\zeta^{\sigma^{-1}(p)})\in \Span C_0$. It follows that the linear span of $C_a$ contains $(ij_{a})(kl_{a})\sigma w$. Since $C_0$ and $C_a$ are contained in $H$, Lemma \ref{lem:i-j-k-l-identity} then tells us that 
%\[
$\sigma^{-1}(i)+\sigma^{-1}(j_a)=\sigma^{-1}(k)+\sigma^{-1}(l_a)\mod p$, proving (\ref{l:basic-C_a-facts::jaia-equation}).
%\]

%As a consequence, we find
%\[
%j_1,\dots,j_d \textrm{\ are\ distinct\ and\ } l_1,\dots,l_d \textrm{\ are\ distinct}.
%\]
%Indeed, assuming by 
To prove (\ref{l:basic-C_a-facts::distinctness}), first let $a,b\in \{1,\dots,d\}$ and assume $j_a=j_b$. From (\ref{l:basic-C_a-facts::jaia-equation}), we know $l_a=l_b$, and so 
%Then we see $\sigma(k)+\sigma(l_1)=\sigma(i)+\sigma(j_1)=\sigma(i)+\sigma(j_2)=\sigma(k)+\sigma(l_2)$ mod $p$, and so $l_1=l_2$ mod $p$. It follows that 
$C_a=(ij_a)(kl_a)C_0=(ij_b)(kl_b)C_0=C_b$, so $a=b$. As for $j_0$, recall that $j_1,\dots,j_d\notin \lambda_M$ and $j_0=k\in\lambda_M$, so they are necessarily distinct.
\end{proof}

We next define a set of pairs
\[
R\subset \{(j,D):j\notin \lambda_M, D \in \{C_0,\dots,C_d\}\}
\]
that will be used to parameterize a subset of edges emanating from the $C_a$. Let $1\leq a\leq d$. Then we define $(j_a,C_0)\in R$. We also define $(j,C_a)\in R$ if there exists $l$ for which $C_a\sim(il)(kj)C_a$ and $\{j,l\}\cap\{j_a,l_a\}=\varnothing$. Consider the map
%Thus, we obtain a map, whose injectivity is due to the uniqueness of the pair $(l,j)$ for each edge out of any fixed vertex:
\[
e\colon R\longhookrightarrow\bigcup_{a=1}^d\{\mathrm{edges\ out\ of\ } C_a\}
\]
defined as follows: $e(j_a,C_0)$ is the edge between $C_a$ and $C_0$; otherwise $e(j,C_a)$ is the edge between $C_a$ and $(il)(kj)C_a$ where $l\notin\lambda_M$ is %the unique element determined by $j$, through Lemma \ref{lem:i-j-k-l-identity}.
uniquely determined by Lemma \ref{l:basic-C_a-facts} (\ref{l:basic-C_a-facts::jaia-equation}). Note that the map $e$ is injective by Lemma \ref{l:unique-edge-labeling-by-j-l}.

\begin{lemma}
\label{l:upper-bnd-form-special-subset-of-edges}
$\sum_{D\sim C_0}d(D) \leq |R|+2d$.
%Furthermore, for each pair $(j,C_a)$ with $1\leq a\leq d$ and $j\notin\lambda_M$, there is at most one value $l$ for which $e(l,j,C_a)\in\E$.
\end{lemma}
\begin{proof}
To prove the lemma, we fix $a\in\{1,\dots,d\}$ and consider every edge out of $C_a$. We show that there are at most $2$ edges out of $C_a$ which are not in the image of the map $e$. Hence, $\sum_{D\sim C_0}d(D)$, which is the total number of edges out of $C_1,\dots,C_d$, is at most $|R|+2d$.

Consider an edge that is not in the image of $e$. Then it is of the form $C_a\sim (il)(kj)C_a$ with $\{j,l\}\cap \{j_a,l_a\}\neq\varnothing$. This breaks up into several cases:\\
{\bf Case 1:} $j=j_a$. If $C_a\sim (il)(kj_a)C_a$, then $l$ is uniquely determined by Lemma \ref{lem:i-j-k-l-identity}. Thus there is at most one edge, out of $C_a$, with $j=j_a$, that is not in the image of $e$.\\
{\bf Case 2:} $l=l_a$. This is similar to Case 1.\\
{\bf Case 3:} $j=l_a$ or $l= j_a$. Then since $C_0=(ij_a)(kl_a)C_a\sim C_a$, and since $j,l$ uniquely determine each other, we must have both $j=l_a$ and $l=j_a$. Thus $(il)(kj)C_a=C_0$ and this edge is equal to $e(j_a,C_0)$, so it is in the image of $e$.

We have therefore shown that for fixed $1\leq a\leq d$, there are at most 2 edges not in the image of the map $e$, corresponding to Cases 1 and 2.
\end{proof}

To complete the proof of Proposition \ref{prop:projection-injective-new}, we need only show $|R|\leq p-m$. This follows from:

\begin{proposition}
\label{prop:proj-map-is-injective}
The projection map
\[
R\lra\,\{1,2,\dots,p\}\setminus\lambda_M
\]
defined by $(j,C_a)\mapsto j$ is injective. 
\end{proposition}

We prove this after a preliminary lemma. For ease of notation, throughout the rest of this section, we let $$j':=\sigma^{-1}(j)$$ for $j\in\{1,\dots,p\}$. Consider the function $f:\{j+p\ZZ:j\neq 2k'-i'\mod p\}\rightarrow \CC$ defined by
\[f(j)=\frac{\zeta^{i'}-\zeta^{j}}{\zeta^{k'}-\zeta^{i'+j-k'}}.\] 
\begin{lemma}
\label{lem:injectivity}
$f$ is injective.
\end{lemma}
\begin{proof}
Note that 
\[f(j')=\frac{\zeta^{i'}-\zeta^{j'}}{\zeta^{k'}-\zeta^{i'+j'-k'}}=\frac{\zeta^{i'}}{\zeta^{k'}}\cdot\frac{1-\zeta^{j'-i'}}{1-\zeta^{i'+j'-2k'}}\]
Note further that since $i\neq k$, we have $i'\neq k'$ and so $-i'\neq i'-2k' \mod p$. Thus, it suffices to show more generally that if $0\leq a,b<p$ with $a\neq b$, then the function 
\[g(x)=\frac{1-\zeta^{a+x}}{1-\zeta^{b+x}}\]
is injective for $x\in\{0,1,\dots,p-1\}\setminus\{p-b\}$. %; indeed, the case $a=b$ corresponds to $i'=k'\mod p$ which does not occur as $i\neq k$. %Also we can assume $x\neq-b$ mod $p$ since $x=b$ corresponds to $l'=i'+j'-k'=k'$ mod $p$, but $l\neq k$.
Now, if $g(x)=g(y)$ for some $x,y\in \{0,1,\dots,p-1\}\setminus\{p-b\}$, then
\[\frac{1-\zeta^{a+x}}{1-\zeta^{b+x}}=\frac{1-\zeta^{a+y}}{1-\zeta^{b+y}}\]
and hence
\[\zeta^{a+y}-\zeta^{a+x}+\zeta^{b+x}-\zeta^{b+y}=0.\]
As a result, if we take the exponents of the polynomial $z^{a+y}-z^{a+x}+z^{b+x}-z^{a+x}$ to be integers between $0$ and $p$ by reducing mod $p$, then it must be the zero polynomial; indeed, it is divisible by the $p$-th cyclotomic polynomial but has degree less than $p$. In particular, the $z^{a+y}$ term must cancel with $z^{a+x}$ or $z^{b+y}$, and hence
\[a+y=a+x\text{ or } a+y=b+y \text{ mod }p.\]
Since $a\neq b$, we see $x=y$, and so $g$ is injective.
\end{proof}

\begin{proof}[Proof of Proposition \ref{prop:proj-map-is-injective}]
Let the $a_J$ and $b_j$ be as in the first few paragraphs of Section \ref{sec:completing-pf}. Consider the binary operation $\odot:\{a_J:J\neq M\}^{\times2}\rightarrow \RR$ defined by
\[a_J\odot a_L=-\frac{a_L-a_M}{a_J-a_M}.\]
%\[a_J\odot a_L=\begin{cases}-{(a_L-1)}/{(a_J-1)}, \text{ if }a_M=1\\-{a_L}/{a_J},\text{ if }a_M=0\end{cases}\]
We will show that if $(j,C_a)\in R$, then
\begin{equation}
\label{eqn:odot-fundamental-eqn}
b_j\odot b_l=f(j'+i'-l_a'),
\end{equation}
where $l$ is the unique element satisfying $l'=i'+j'-k'$. Assuming this for the moment, we see $j$ determines $l$, which then determines $b_j\odot b_l=f(j'+i'-l_a')$. Since $f$ is injective by Lemma \ref{lem:injectivity}, we see $j$ determines $l'_a$. Since $l_0,\dots,l_d$ are distinct, by Lemma \ref{l:basic-C_a-facts} (\ref{l:basic-C_a-facts::distinctness}), we find that there is at most one value $0\leq a\leq d$ for which $(j,C_a)\in R$, thereby proving the proposition.

It remains to prove (\ref{eqn:odot-fundamental-eqn}). We first consider elements of form $(j_a,C_0)\in R$. In this case, $j=j_a$, $l=l_a$, and $b_i=b_k=a_M\notin\{b_j,b_l\}$. Since $H$ contains both $C_0$ and $C_a=(ij_a)(kl_a)C_0$, Lemma \ref{lem:i-j-k-l-identity} shows that $i'+j'=k'+l'$ mod $p$. Since $\sigma w\in \Span (C_0)\subset H$ and $(ij)(kl)\sigma w\in\Span((ij)(kl)C_0)\subset H$, we find
%In this case, $H$ contains both $C_0$ and $(ij_a)(kl_a)C_0$. Let $j=j_a,l=l_a$ and note that $b_i=b_k=a_M\notin\{b_j,b_l\}$. So, by Lemma \ref{lem:i-j-k-l-identity}, we have $i'+j'=k'+l'$ mod $p$. Since $\sigma w\in \Span (C_0)\subset H$ and $(ij)(kl)\sigma w\in\Span((ij)(kl)C_0)\subset H$, we find
\[(\dots,\zeta^{j'}-\zeta^{i'},\dots,\zeta^{i'}-\zeta^{j'},\dots,\zeta^{l'}-\zeta^{k'},\dots,\zeta^{k'}-\zeta^{l'},\dots)=\sigma w-(ij)(kl)\sigma w\in H\]
where the omitted entries are 0, and the non-zero entries are in the $j,i,l,k$-th positions, respectively. As $H=(b_1,\dots,b_p)^\perp$, we have
\[b_{j}(\zeta^{j'}-\zeta^{i'})+b_i(\zeta^{i'}-\zeta^{j'})+b_l(\zeta^{l'}-\zeta^{k'})+b_k(\zeta^{k'}-\zeta^{l'})=0\]
%Then\[b_{j}(\zeta^{j'}-\zeta^{i'})+b_l(\zeta^{l'}-\zeta^{k'})=\zeta^{j'}-\zeta^{i'}+\zeta^{l'}-\zeta^{k'}\]and hence
and so
\[b_{l}=\frac{(\zeta^{j'}-\zeta^{i'}+\zeta^{l'}-\zeta^{k'})a_M-b_{j}(\zeta^{j'}-\zeta^{i'})}{\zeta^{l'}-\zeta^{k'}}.\]
Since $l'=i'+j'-k'$ mod $p$, we have $\zeta^{l'}=\zeta^{i'+j'-k'}$. Note that $j'\neq 2k'-i'$ mod $p$ since otherwise we would have $k'=l'$ mod $p$, which is not possible as $k\neq l$. As a result, $f(j')$ is well-defined and
\[
b_l=a_M+(a_M-b_j)f(j').
\]
%\[b_l=f(j')-b_jf(j')+1.\]{\bf Case 2:} $b_i=b_k=a_M=0$. Then\[b_{j}(\zeta^{j'}-\zeta^{i'})+b_l(\zeta^{l'}-\zeta^{k'})=0\]and so\[b_l=-b_jf(j').\]In either case, $b_j\odot b_l=f(j')=f(j'+i'-l_0')$, as desired.\\
As a result, we have our desired equality
\[
b_j\odot b_l=-\frac{b_l-a_M}{b_j-a_M}=f(j')=f(j'+i'-l_0').
\]

We next consider an element of the form $(j,C_a)\in R$ for some $1\leq a\leq d$. Then, by definition, we have $C_a\sim (il)(kj)C_a$ for some $\{j,l\}\cap \{j_a,l_a\}=\varnothing$. 
Let $w_a:=(ij_a)(kl_a)\sigma w\in C_a \subset H$ and note $(il)(kj)w_a\in (ij)(kl)C_a\subset H$. So,
\[(\dots,\zeta^{j_a'}-\zeta^{l'},\dots,\zeta^{l_a'}-\zeta^{j'},\dots.,\zeta^{j'}-\zeta^{l_a'},\dots,\zeta^{l'}-\zeta^{j_a'},\dots)=w_a-(il)(kj)w_a\in H\]
where the omitted entries are 0, and the non-zero entries are in the $i,k,j,l$-th position, respectively. As a result,
\[b_i(\zeta^{j_a'}-\zeta^{l'})+b_k(\zeta^{l_a'}-\zeta^{j'})+b_j(\zeta^{j'}-\zeta^{l_a'})+b_l(\zeta^{l'}-\zeta^{j_a'})=0.\]
%Let $h:\{j+p\ZZ:j\neq 2j_a'-l_a' \mod p\}\rightarrow \CC$ be given by
%\[h(j)=\frac{\zeta^{l_a'}-\zeta^{j}}{\zeta^{j_a'}-\zeta^{l_a'+j-j_a'}}.\]
%{\bf Case 1:} $b_i=b_k=a_M=1$. Then
%\[b_{j}(\zeta^{j'}-\zeta^{l_a'})+b_l(\zeta^{l'}-\zeta^{j_a'})=\zeta^{l'}-\zeta^{j_a'}+\zeta^{j'}-\zeta^{l_a'}\]
It follows that
\[b_{l}=\frac{(\zeta^{l'}-\zeta^{j_a'}+\zeta^{j'}-\zeta^{l_a'})a_M-b_{j}(\zeta^{j'}-\zeta^{l_a'})}{\zeta^{l'}-\zeta^{j_a'}}\]
and hence
\[
b_j\odot b_l=\frac{\zeta^{j'}-\zeta^{l_a'}}{\zeta^{l'}-\zeta^{j_a'}}.
\]
It remains to prove this expression equals $f(j'+i'-l_a')$.

%Note that $j'\neq 2j_a'-l_a'$ mod $p$, since otherwise we would have $j_a'=l_a'$ mod $p$, but $j_a\neq l_a$. So $h(j')$ is well-defined and\[b_l=h(j')-b_jh(j')+1.\]{\bf Case 2:} $b_i=b_k=a_M=0$. Then \[b_{j}(\zeta^{j'}-\zeta^{l_a'})+b_l(\zeta^{l'}-\zeta^{j_a'})=0\] and so \[b_l=-b_jh(j').\] In either case, $b_j\odot b_l=h(j')$. 
%
%To finish the proof, it remains to show $h(j')=f(j'+i'-l_a')$. %Since $H$ contains both $C_0$ and $C_a$, $i'+j_a'=k'+l_a'\mod p$ by Lemma \ref{lem:i-j-k-l-identity}. 
Since $i'+j_a'=k'+l_a'$ mod $p$ 
%\kent{it seems $i'+j'=l'+k'$ mod $p$ is all we need here} 
and $i'+j'=l'+k'$ mod $p$, we have
\[l'-l_a'+i'=i'+(j'+i'-l_a')-k'\mod p.\]
As a result, 
%\[{h(j')}=\frac{\zeta^{l_a'}-\zeta^{j'}}{\zeta^{j_a'}-\zeta^{l'}}=\frac{\zeta^{i'}-\zeta^{j'-l_a'+i'}}{\zeta^{j_a'-l_a'+i'}-\zeta^{l'-l_a'+i'}}=\frac{\zeta^{i'}-\zeta^{j'-l_a'+i'}}{\zeta^{k'}-\zeta^{i'+(j'+i'-l_a')-k'}}=f(j'+i'-l_a')\]
\[
f(j'+i'-l_a')=\frac{\zeta^{i'}-\zeta^{j'-l_a'+i'}}{\zeta^{k'}-\zeta^{i'+(j'+i'-l_a')-k'}}=\frac{\zeta^{i'}-\zeta^{j'-l_a'+i'}}{\zeta^{j_a'-l_a'+i'}-\zeta^{l'-l_a'+i'}}=\frac{\zeta^{l_a'}-\zeta^{j'}}{\zeta^{j_a'}-\zeta^{l'}}
\]
thereby finishing the proof.
%, we find $j_a'-l_a'+i'=k'$ mod $p$
\end{proof}

\bibliographystyle{alpha}
\bibliography{refs}
\def\cprime{$'$}
\def\cprime{$'$}
\end{document}

\end{document}